\documentclass{amsart}
\usepackage[margin=3.9cm]{geometry}

\usepackage{amssymb}
\usepackage{graphicx} 
\usepackage{mathrsfs}
\usepackage{enumerate}
\usepackage{xspace}
\usepackage{color}
\usepackage{enumerate}
\usepackage{enumitem}
\usepackage{amsthm}
\usepackage{dsfont}
\usepackage{bbm}
\usepackage[colorlinks=true, linkcolor = blue, citecolor = blue]{hyperref}
\usepackage[numbers]{natbib}
\usepackage[backgroundcolor=white,bordercolor=red]{todonotes}
\DeclareMathAlphabet{\mathpzc}{OT1}{pzc}{m}{it}
\usepackage[nameinlink]{cleveref}
\numberwithin{equation}{section}

\makeatletter
\def\@settitle{\begin{center}%
		\baselineskip14\p@\relax
		\bfseries
		\uppercasenonmath\@title
		\@title
		\ifx\@subtitle\@empty\else
		\\[1ex]\uppercasenonmath\@subtitle
		\footnotesize\mdseries\@subtitle
		\fi
	\end{center}%
}
\def\subtitle#1{\gdef\@subtitle{#1}}
\def\@subtitle{}
\makeatother

\begin{document}

\expandafter\let\expandafter\oldproof\csname\string\proof\endcsname
\let\oldendproof\endproof
\renewenvironment{proof}[1][\proofname]{%
	\oldproof[\scshape\hspace{1em}#1]%
}{\oldendproof}

\newtheoremstyle{mystyle_thm}
{6pt}
{6pt}
{\itshape}
{1em}
{\scshape}
{.}
{.5em}
{}%

\newtheoremstyle{mystyle_def}
{6pt}
{6pt}
{}
{1em}
{\scshape}
{.}
{.5em}
{}%

\theoremstyle{mystyle_thm}

\newtheorem{theorem}{Theorem}[section]
\newtheorem{lemma}[theorem]{Lemma}
\newtheorem{proposition}[theorem]{Proposition}
\newtheorem{corollary}[theorem]{Corollary}
\newtheorem{Ass}[theorem]{Assumption}
\newtheorem{condition}[theorem]{Condition}
\newtheorem{definition}[theorem]{Definition}

\theoremstyle{mystyle_def}

\newtheorem{example}[theorem]{Example}
\newtheorem{remark}[theorem]{Remark}
\newtheorem{SA}[theorem]{Standing Assumption}
\newtheorem*{discussion}{Discussion}
\newtheorem{remarks}[theorem]{Remark}
\newtheorem*{notation}{Remark on Notation}
\newtheorem{application}[theorem]{Application}

\newcommand{\of}{[\hspace{-0.06cm}[}
\newcommand{\gs}{]\hspace{-0.06cm}]}

\newcommand\llambda{{\mathchoice
		{\lambda\mkern-4.5mu{\raisebox{.4ex}{\scriptsize$\backslash$}}}
		{\lambda\mkern-4.83mu{\raisebox{.4ex}{\scriptsize$\backslash$}}}
		{\lambda\mkern-4.5mu{\raisebox{.2ex}{\footnotesize$\scriptscriptstyle\backslash$}}}
		{\lambda\mkern-5.0mu{\raisebox{.2ex}{\tiny$\scriptscriptstyle\backslash$}}}}}

\newcommand{\1}{\mathds{1}}

\newcommand{\F}{\mathbf{F}}
\newcommand{\G}{\mathbf{G}}

\newcommand{\B}{\mathbf{B}}

\newcommand{\M}{\mathcal{M}}

\newcommand{\la}{\langle}
\newcommand{\ra}{\rangle}

\newcommand{\lle}{\langle\hspace{-0.085cm}\langle}
\newcommand{\rre}{\rangle\hspace{-0.085cm}\rangle}
\newcommand{\blle}{\Big\langle\hspace{-0.155cm}\Big\langle}
\newcommand{\brre}{\Big\rangle\hspace{-0.155cm}\Big\rangle}

\newcommand{\X}{\mathsf{X}}

\newcommand{\tr}{\operatorname{tr}}
\newcommand{\N}{{\mathbb{N}}}
\newcommand{\cadlag}{c\`adl\`ag }
\newcommand{\on}{\operatorname}
\newcommand{\oP}{\overline{P}}
\newcommand{\oO}{\mathcal{O}}
\newcommand{\D}{\mathsf{D}} 
\newcommand{\bx}{\mathsf{x}}
\newcommand{\bb}{\hat{b}}
\newcommand{\bs}{\hat{\sigma}}
\newcommand{\bv}{\hat{v}}
\renewcommand{\v}{\mathfrak{m}}
\newcommand{\ob}{\bar{b}}
\newcommand{\oa}{\bar{a}}
\newcommand{\os}{\widehat{\sigma}}
\renewcommand{\j}{\varkappa}
\newcommand{\scl}{\ell}
\newcommand{\Y}{\mathscr{Y}}
\newcommand{\Z}{\mathscr{Z}}
\newcommand{\T}{\mathcal{T}}
\newcommand{\con}{\mathsf{c}}
\newcommand{\nk}{\hspace{-0.25cm}{{\phantom A}_k^n}}
\newcommand{\nl}{\hspace{-0.25cm}{{\phantom A}_1^n}}
\newcommand{\nm}{\hspace{-0.25cm}{{\phantom A}_2^n}}
\newcommand{\n}{\hspace{-0.35cm}{\phantom {Y_s}}^n}
\newcommand{\nme}{\hspace{-0.35cm}{\phantom {Y_s}}^{n - 1}}
\renewcommand{\o}{\hspace{-0.35cm}\phantom {Y_s}^0}
\newcommand{\e}{\hspace{-0.4cm}\phantom {U_s}^1}
\newcommand{\z}{\hspace{-0.4cm}\phantom {U_s}^2}
\newcommand{\iii}{|\hspace{-0.05cm}|\hspace{-0.05cm}|}
\newcommand{\co}{\overline{\on{co}}}
\renewcommand{\k}{\mathsf{k}}
\newcommand{\ovb}{\overline{b}}
\newcommand{\ova}{\overline{a}}
\newcommand{\s}{\mathfrak{s}}
\newcommand{\opsi}{\overline{\Psi}}
\newcommand{\ol}{\mathcal{L}}
\newcommand{\cW}{\mathscr{W}}
\newcommand{\cU}{\mathcal{U}}
\newcommand{\oD}{\overline{D}}
\newcommand{\ua}{\underline{a}}
\newcommand{\ou}{\overline{b}}
\newcommand{\uu}{\underline{b}}

\renewcommand{\epsilon}{\varepsilon}
\renewcommand{\rho}{\varrho}

\newcommand{\fPs}{\fP_{\textup{sem}}}
\newcommand{\fPas}{\mathfrak{S}_{\textup{ac}}}
\newcommand{\rrarrow}{\twoheadrightarrow}
\newcommand{\cA}{\mathcal{A}}
\newcommand{\ocA}{\mathcal{U}}
\newcommand{\cR}{\mathcal{R}}
\newcommand{\cK}{\mathcal{K}}
\newcommand{\cQ}{\mathcal{Q}}
\newcommand{\cF}{\mathcal{F}}
\newcommand{\cE}{\mathcal{E}}
\newcommand{\cC}{\mathcal{C}}
\newcommand{\cD}{\mathcal{D}}
\newcommand{\bC}{\mathbb{C}}
\newcommand{\cH}{\mathcal{H}}
\newcommand{\bth}{\overset{\leftarrow}\theta}
\renewcommand{\th}{\theta}
\newcommand{\cG}{\mathcal{G}}
\newcommand{\fPasn}{\mathfrak{S}^{\textup{ac}, n}_{\textup{sem}}}
\newcommand{\CLM}{\mathfrak{M}^\textup{ac}_\textup{loc}}
\newcommand{\Sd}{\mathcal{S}^\textup{sp}_{\textup{d}}}
\newcommand{\Sc}{\mathcal{S}}
\newcommand{\Sac}{\mathcal{S}_\textup{ac}}
\newcommand{\A}{\mathsf{A}}
\newcommand{\Td}{\mathsf{T}^\textup{d}}
\renewcommand{\t}{\mathfrak{t}}

\newcommand{\bR}{\mathbb{R}}
\newcommand{\nnabla}{\nabla}
\newcommand{\f}{\mathfrak{f}}
\newcommand{\g}{\mathfrak{g}}
\newcommand{\oconv}{\overline{\on{co}}\hspace{0.075cm}}
\renewcommand{\a}{\mathfrak{a}}
\renewcommand{\b}{\mathfrak{b}}
\renewcommand{\d}{\mathsf{d}}
\newcommand{\bS}{\mathbb{S}^d_+}
\newcommand{\p}{\mathsf{p}}
\newcommand{\dr}{r} 
\newcommand{\m}{\mathfrak{m}}
\newcommand{\Q}{Q}
\newcommand{\usc}{\textit{USC}}
\newcommand{\lsc}{\textit{LSC}}
\newcommand{\q}{\mathfrak{q}}
\renewcommand{\X}{\mathscr{X}}
\newcommand{\W}{\mathscr{W}}
\newcommand{\fP}{\mathcal{P}}
\newcommand{\w}{\mathsf{w}}
\newcommand{\oM}{\mathsf{M}}
\newcommand{\oZ}{\mathsf{Z}}
\newcommand{\oK}{\mathsf{K}}
\renewcommand{\Re}{\operatorname{Re}}
\newcommand{\cCk}{\mathsf{c}_k}
\newcommand{\C}{\mathsf{C}}
\newcommand{\cP}{\mathcal{P}}
\newcommand{\oPi}{\overline{\Pi}}
\newcommand{\cI}{\mathcal{I}}
\renewcommand{\P}{\mathbf{P}}
\renewcommand{\X}{\mathsf{X}}
\renewcommand{\Q}{\mathsf{Q}}
\renewcommand{\Y}{\mathsf{Y}}
\renewcommand{\Z}{\mathsf{Z}}
\newcommand{\E}{\mathbf{E}}
\renewcommand{\Q}{\mathbf{Q}}

\renewcommand{\emptyset}{\varnothing}

\allowdisplaybreaks

\makeatletter
\@namedef{subjclassname@2020}{%
	\textup{2020} Mathematics Subject Classification}
\makeatother

 \title[Representation Property for general diffusion semimartingales]{On the Representation Property for 1d general diffusion semimartingales} 
\author[D. Criens]{David Criens}
\address{D. Criens - Albert-Ludwigs-University of Freiburg, Ernst-Zermelo-Str. 1, 79104 Freiburg, Germany.}
\email{david.criens@stochastik.uni-freiburg.de}

\author[M. Urusov]{Mikhail Urusov}
\address{M. Urusov - University of Duisburg-Essen, Thea-Leymann-Str. 9, 45127 Essen, Germany.}
\email{mikhail.urusov@uni-due.de}



\thanks{}
\date{\today}

\maketitle

\begin{abstract}
A general diffusion semimartingale is a one-dimensional path-continuous semimartingale that is also a regular strong Markov process. We say that a continuous semimartingale has the representation property if all local martingales w.r.t. its canonical filtration have an integral representation w.r.t. its continuous local martingle part.
The representation property is of fundamental interest in the field of mathematical finance, where it is strongly connected to market completeness.
The main result from this paper shows that the representation property holds for a general diffusion semimartingale (that is not started in an absorbing boundary point) if and only if its scale function is
(locally)
absolutely continuous on the interior of the state space.
As an application of our main theorem, we deduce that
the laws of
general diffusion semimartingales with such scale functions are extreme points of their semimartingale problems,
and, moreover, we construct a general diffusion semimartingale whose law is \emph{no} extreme point of its semimartingale problem. These observations contribute to a solution of problems posed by J.~Jacod and M.~Yor, and D.~W.~Stroock and M.~Yor, on the extremality of strong Markov solutions.
\end{abstract}

\section{Introduction}
A (one-dimensional) \emph{general diffusion}
is a time- and path-continuous regular strong Markov process (cf., e.g., \cite{breiman1968probability,itokean74,RY,RW2}).
The class of general diffusions is very rich. For example, solutions to stochastic differential equations (SDEs) under the famous Engelbert--Schmidt conditions (\cite{engel_schmidt_3}) are general diffusions, but also sticky and skew Brownian motions, or certain processes with
(instantaneously or slowly)
reflecting boundaries, see also Examples \ref{ex: SDE}, \ref{ex: sticky BM} and~\ref{ex: chain} below. It is well-known that general diffusions are characterized by two deterministic objects: the scale functions and the speed measure. In the following, we call them the {\em diffusion characteristics}. A particularly interesting subclass of general diffusions are those which are semimartingales and we call them {\em general diffusion semimartingales}. The semimartingale property of general diffusions was studied profoundly in the fundamental work \cite{CinJPrSha} by E.~\c Cinlar, J.~Jacod, P.~Protter and M.~J.~Sharpe.

\smallskip 
In this paper we answer the question whether general diffusion semimartingales have the representation property as defined in the seminal monograph \cite{JS} by J.~Jacod and A.~N.~Shiryaev, i.e., whether for a given general diffusion semimartingale all local martingales w.r.t. its canonical filtration have an integral representation w.r.t. its continuous local martingale part. The main result from this article shows that the representation property is equivalent to
(local)
absolute continuity of the scale function on the interior of the state space, which appears to be a simple deterministic characterization.
This generalizes Theorem~7 in \cite{EngelbertHess1981}, where it is shown by a different method that general diffusions on natural scale without reflecting boundaries always have the representation property.
Based on our main result, we also provide an example of a general diffusion semimartingale that fails to have the representation property, see Example~\ref{ex: no RP} below.

\smallskip 
Our interest in this question is motivated by mathematical finance. More specifically, as indicated in the monographs of M.~Jeanblanc, M.~Yor and M.~Chesney \cite{CJY}, G.~Kallianpur and R.~L.~Karandikar \cite{KalliKara} or A.~N.~Shiryaev \cite{shir}, or the papers by C.~Kardaras and J.~Ruf~\cite{kardarasruf} or A.~N.~Shiryaev and A.~S.~Cherny \cite{CherShir}, the representation property is intrinsically connected to completeness of an associated financial market model. In our previous paper \cite{criensurusov_23}, we already established precise characterizations of the absence of arbitrage in the sense of NA (``no arbitrage''), NUBPR (``no unbounded profit with bounded risk'') and NFLVR (``no free lunch with vanishing risk'') for general diffusion semimartingale markets in terms of the diffusion characteristics. We think that the main result from this paper provides a first important step into the direction of a deterministic characterization of market completeness. 

\smallskip 
The representation property is closely connected to so-called semimartingale problems as studied in the monograph \cite{JS}. A probability measure \(\Q\) on the Wiener space is said to be a solution to the semimartingale problem associated to a pair \((B, C)\) of candidate semimartingale characteristics if the coordinate process is a \(\Q\)-semimartingale with the candidate \((B, C)\) as its semimartingale characteristics. The convex set of solutions is denoted by \(\mathscr{S} (B, C)\). 
It is proved in \cite{JS} that under \(\Q \in \mathscr{S} (B, C)\) the coordinate process satisfies the representation property
and the initial \(\sigma\)-field is \(\Q\)-trivial
if and only if \(\Q\) is an extreme point in \(\mathscr{S} (B, C)\). Consequently, recalling also Blumenthal's zero-one law, our main result shows that a general diffusion semimartingale, which is not started in an absorbing boundary point, is an extreme point of the corresponding semimartingale problem if and only if its scale function is
(locally)
absolutely continuous on the interior of the state space. At this point, we stress that semimartingale problems associated to general diffusion semimartingales may have infinitely many solutions. For example, this is the case for the semimartingale problem of a sticky Brownian motion, see Example~\ref{ex: SBM} below. 

The problem to identify strong Markov solutions as extreme points of certain non-well-posed (semi)martingale problems was already considered by J.~Jacod and M.~Yor \cite{JY_77}. For the special case of Girsanov's SDE, they proved that all strong Markov solutions are extreme points. 
The question whether such a result holds also for more general semimartingale problems, possibly multidimensional and with jumps, was left open. 
A related problem was considered by D.~W.~Stroock and M.~Yor \cite{StrYor} in the context of Markovian selections to non-well-posed (semi)martingale problems. 
We present a counterexample showing that there are strong Markov solutions that are {\em no} extreme points, see Example~\ref{ex: no RP}.

\smallskip
The remainder of this paper is organized as follows. In Section~\ref{sec: main} we provide our main results. The proofs are given in Section~\ref{sec: pf}.

\section{The Representation Property for General Diffusion Semimartingales} \label{sec: main}

To ease our presentation, we will pose ourselves in the canonical diffusion setting on a path space (cf., e.g., \cite[Definitions V.45.1 and~V.45.2]{RW2}).
Let \(J \subset \bR\) be a bounded or unbounded, closed, open or half-open interval. 
We define \(\Omega\triangleq C(\mathbb R_+; \bR)\), the space of continuous functions \(\mathbb{R}_+ \to \bR\). The coordinate process on \(\Omega\) is denoted by \(\X\), i.e., \(\X_t (\omega) = \omega(t)\) for \(t \in \mathbb{R}_+\) and \(\omega \in \Omega\). 
We also set \(\mathcal{F} \triangleq \sigma (\X_s, s \geq 0)\) and \(\mathcal{F}_t \triangleq \bigcap_{s > t}\sigma (\X_r, r \leq s)\) for all~\(t \in \mathbb{R}_+\). 

A map \((J \ni x \mapsto \P_x)\), from \(J\) into the set of probability measures on \((\Omega, \mathcal{F})\), is called a
\emph{regular continuous strong Markov proccess} or a
\emph{general diffusion}
(with state space $J$) if the following hold:
\begin{enumerate}
	\item[\textup{(i)}] $\P_x(\X_0 = x) = \P_x(C(\mathbb R_+;J))=1$ for all $x\in J$;
	\item[\textup{(ii)}] the map \(x \mapsto \P_x(A)\) is measurable for all \(A \in \mathcal{F}\);
	\item[\textup{(iii)}] the \emph{strong Markov property} holds, i.e., for any stopping time \(\tau\) and any \(x \in J\), the kernel \(\P_{\X_\tau}\) is the regular conditional \(\P_x\)-distribution of \((\X_{t + \tau})_{t \geq 0}\) given \(\mathcal{F}_{\tau}\) on \(\{\tau < \infty\}\), i.e., \(\P_x\)-a.s. on~\(\{\tau < \infty\}\)
	\[  
	\P_x ( \X_{\cdot + \tau} \in d\omega \mid \mathcal{F}_\tau) = \P_{\X_\tau} (d \omega);
	\]
	\item[\textup{(iv)}] the diffusion is \emph{regular}, i.e., for all $x\in J^\circ$ and $y\in J$,
	\[
	\P_x(T_y<\infty)>0,
	\]
	where $T_y\triangleq\inf\{t\ge0:\X_t=y\}$ (with the usual convention $\inf\emptyset\triangleq\infty$).
\end{enumerate} 

It is well-known that a general diffusion \((x \mapsto \P_x)\) is characterized by two deterministic objects, the scale function \(\s\) and the speed measure \(\m\). The former is a strictly increasing continuous function from \(J\) into \(\bR\), while the latter is a measure on \((J, \mathcal{B} (J))\) that is locally finite on \((J^\circ, \mathcal{B} (J^\circ))\) and strictly positive, i.e., \(\m ([a, b]) > 0\) for all \(a, b \in J^\circ\) with \(a < b\). The pair \((\s, \m)\) is sometimes called the ``characteristics of the general diffusion''. For precise definitions and more details on these concepts we refer to the seminal monograph \cite{itokean74} by K.~It\^o and H.~P.~McKean. More gentle introductions can be found in \cite{breiman1968probability,freedman,kallenberg,RY,RW2}.

\begin{example}[SDEs under the Engelbert--Schmidt conditions] \label{ex: SDE}
Let \(\Y\) be a (possibly explosive)
unique in law weak
solution to the SDE
	\[
	d \Y_t = \mu (\Y_t) dt + \sigma (\Y_t) d W_t, \quad \Y_0 = x_0, 
	\]
	where \(\mu \colon J^\circ \to \bR\) and \(\sigma \colon J^\circ \to \bR\) satisfy the Engelbert--Schmidt conditions 
\[
\forall x\in J^\circ\colon
\sigma(x)\ne0
\qquad\text{and}\qquad
\frac{1+|\mu|}{\sigma^2} \in L^1_\textup{loc} (J^\circ),
\]
and the accessible boundary points are stipulated to be absorbing.
We refer to the paper \cite{engel_schmidt_3} or the monograph \cite[Chapter 5.5]{KaraShre} for a detailed discussion of SDEs under the Engelbert--Schmidt conditions. 
	
	It is well-known (\cite[Corollary~4.23]{engel_schmidt_3}) that SDEs under the Engelbert--Schmidt conditions give rise to a regular strong Markov family (possibly with a state space \(J \subset [- \infty, + \infty]\)) with scale function
			\begin{align*}
		\s (x) = \int^x \exp \Big\{ - \int^y \frac{2 \mu (z)}{\sigma^2 (z)} \, dz \, \Big\} \, dy, \quad x \in J^\circ,
	\end{align*} 
and speed measure 
\[
\m (dx) = \frac{dx}{\s' (x) \sigma^2 (x)} \text{ on } J^\circ.
\]
At this point, we stress that \(\Y\) is not necessarily a semimartingale. The simplest possible problem is that \(\Y\) may reach \(\infty\) or \(- \infty\) in finite time with positive probability.
We refer to \cite{MU15} for a thorough discussion of the semimartingale property of explosive solutions to~SDEs.
\end{example}

\begin{example}[Sticky Brownian motion] \label{ex: sticky BM}
Another interesting class of diffusions are SDEs with stickiness. The most prominent example is the sticky Brownian motion\footnote{More precisely,
Brownian motion with state space $\bR$, sticky at zero.},
which is the solution \(\Y\) to the system 
	\[
	d \Y_t = \1_{\{ \Y_t \not = 0\}} d W_t, \quad \1_{\{\Y_t = 0\}} dt = \rho\, d L^0_t (\Y),
	\]
	where \(\rho > 0\) is a so-called stickiness parameter and \(L^0 (\Y)\) is the semimartingale local time of the solution \(\Y\) in zero. For a discussion of this representation we refer to the paper \cite{EngPes}.
	
	The sticky Brownian motion is a general diffusion on natural scale, i.e., with the identity as scale function,
with state space $\bR$
and speed measure 
	\[
	\m (dx) = dx  + \rho \, \delta_0 (dx).	
	\]
At this point, we notice that the sticky Brownian motion cannot be realized as a solution to an SDE like in Example~\ref{ex: SDE}, as the speed measure in Example~\ref{ex: SDE} is always absolutely continuous w.r.t. the Lebesgue measure.
It is clear that one can also consider diffusions with more than one 
sticky point.
\end{example}

\begin{example}[General diffusions with a countable dense set of sticky points] \label{ex: chain}
	The following example is due to W.~Feller and H.~P.~McKean, see \cite[Section~2.12]{freedman}.
	Let \(D = \{d_1, d_2, \dots \}\) be a countable dense subset of \(\bR\) and let \(\m\) be a measure on \((\bR, \mathcal{B}(\bR))\) that is concentrated on \(D\) and such that \(\m (\{d_k\}) > 0\) and \(\sum_{k = 1}^\infty \m (\{d_k\}) < \infty\). Then, \(\m\) is a valid speed measure. Let \((\bR\ni x \mapsto \P_x)\) be the general diffusion on natural scale with
state space $\bR$ and
speed measure \(\m\).
While this process is continuous, it has quite peculiar paths in the sense that it spends positive Lebesgue time in every point $d_k$ (without intervals of constancy) and zero Lebesgue time in every point in $\bR\setminus D$. Moreover,
one can prove that \(\P_x (\X_t \in D) = 1\) for all
\(x\in\bR\) and \(t>0\),
cf. \cite[Lemma~144, p.~163]{freedman}
or Section~4.11 in \cite{itokean74}.
\end{example}

Our inputs are a state space $J$, a scale function $\s$ and a speed measure $\m$.
They define a general diffusion $(J\ni x\mapsto\P_x)$.
In what follows we say that
$x_0\in J$
is \emph{non-absorbing} if
$x_0\in J^\circ$
or $x_0$ is an (instantaneously or slowly) reflecting boundary point of $J$.
We work under the following standing assumption. 

\begin{SA} \label{SA: 1}
	For a fixed non-absorbing initial value \(x_0 \in J\), the underlying filtered probability space \((\Omega, \cF, (\cF_t)_{t \geq 0}, \P_{x_0})\) is the
canonical setup introduced above,
and the coordinate process \(\X\) is a semimartingale. 
\end{SA}

\begin{remark}
It is well-known that
the requirement for $\X$ to be a semimartingale in
Standing Assumption~\ref{SA: 1} is merely a condition on the inverse scale function
	\[
	\q \triangleq \s^{-1}.
	\]
	In particular, Standing~Assumption~\ref{SA: 1} implies that the restriction \(\q|_{\s(J^\circ)}\) is the difference of two convex functions on \(\s(J^\circ)\), which entails that its right- and left derivatives exist. We denote them by 
	\[
	\q'_+ (x) \triangleq \lim_{h \searrow 0} \frac{\q (x + h) - \q (x)}{h}, \qquad \q'_- (x) \triangleq \lim_{h \searrow 0} \frac{\q (x) - \q (x - h)}{h}, \qquad x \in \s(J^\circ). 
	\]
	It is well-known that \(\{\q'_+ \not = \q'_-\}\) is at most countable (cf. \cite[Section~3.16]{freedman}).
	For more details on the precise properties of \(\q\) that are implied by Standing~Assumption~\ref{SA: 1}, we refer to \cite[Section~5]{CinJPrSha}. 
\end{remark}

\begin{definition}[Representation property]
We say that the {\em representation property (RP)} holds if every local martingale \(\mathsf{M} = (\mathsf{M}_t)_{t \geq 0}\) has a representation 
\[
\mathsf{M} = \mathsf{M}_0 + \int_0^\cdot H_s d \X^c_s, 
\]
where \(H \in L^2_\textup{loc} (\X^c)\) and \(\X^c\) is the continuous local martingale part of \(\X\).
\end{definition}

We are in the position to present our main result. Its proof is given in Section~\ref{sec: pf} below.

\begin{theorem} \label{theo: complete}
	The following are equivalent:
	\begin{enumerate}
	\item[\textup{(a)}] The RP holds.
	\item[\textup{(b)}] \(\mu_L ( \{ x \in \s (J^\circ) \colon \q'_+ (x) = 0 \}) = 0\), where \(\mu_L\) denotes the Lebesgue measure.
	\item[\textup{(c)}] The scale function \(\s\) is absolutely continuous on all compact subintervals of~$J^\circ$.
	\end{enumerate}
\end{theorem} 

\begin{remark}[Role of the initial value]
As (b) and (c) from Theorem~\ref{theo: complete} are independent of the generic initial value \(x_0 \in J\) from Standing Assumption~\ref{SA: 1}, the RP holds for all non-absorbing
\(x_0 \in J\) if and only if it holds for some non-absorbing \(x_0 \in J\). In case \(x_0 \in J \setminus J^\circ\) is an absorbing boundary point, the RP always holds, in particular, independently of the scale function, because then all local martingales on \((\Omega, \cF, (\cF_t)_{t \geq 0}, \P_{x_0})\) are constant. 
\end{remark}

\begin{example}
Theorem~\ref{theo: complete} yields that,
for the SDE setting of Example~\ref{ex: SDE}
as well as
for the sticky models from Examples \ref{ex: sticky BM} and~\ref{ex: chain}, the RP holds.
In fact, for what concerns Example~\ref{ex: SDE}, the fact that the RP holds is well-known and follows, e.g., from \cite[Theorem~III.4.29]{JS}, as the uniqueness in law property of the SDE from Example~\ref{ex: SDE} implies uniqueness of the solution of the associated semimartingale problem.
Notice, however, that such an argumentation does not apply to arbitrary general diffusions. For example, below we observe that there is {\em no uniqueness} of the semimartingale problem for the sticky Brownian motion. Let us also mention that Examples~\ref{ex: sticky BM} and \ref{ex: chain} are covered by Theorem~7 in \cite{EngelbertHess1981}. Below, we discuss a negative result that is not in the scope of \cite{EngelbertHess1981}.
\end{example}

In the previous example, we discussed a variety of positive situations where the RP holds. It is also interesting to provide a negative result, i.e., to give an example where the RP fails. 

\begin{example}[General diffusion semimartingale for which the RP fails] \label{ex: no RP}
	We now provide an example of a general diffusion semimartingale that fails to have the RP. 

\begin{lemma}\label{lem:140624a1}
	There exists a strictly increasing continuous function \(\q \colon [0, 1] \to [0, 1]\) that is continuously differentiable on \((0, 1)\) with a derivative \(\q'\) that extends to a Lipschitz continuous function on \([0, 1]\) and that has the property \(\mu_L ( \{z \in (0, 1) \colon \q' (z) = 0\} ) > 0\). 
\end{lemma}

We thank Oleg Dragoshansky \cite{OD} for explaining us another interesting example for a function~$\q$ as in Lemma~\ref{lem:140624a1}.

\begin{proof} [Proof of Lemma~\ref{lem:140624a1}]
Let \(G\subset [0, 1]\) be any closed set without inner points such that \(\mu_L (G) > 0\).
For example, this could be a set as discussed in \cite[Example~1.7.6, p.~30]{bogachev} or a \emph{generalized} (or \emph{fat}) Cantor set, cf. \cite[p.~39]{folland}.
We define 
	\[
	\q (x) \triangleq \int_0^x d_G (z) dz, \quad d_G (x) \triangleq \inf_{y \in G} |y - x|, \quad x \in [0, 1].
	\] 
	Evidently, \(\q\) is differentiable with Lipschitz continuous derivative \(\q' = d_G\). Furthermore, as $G$ is closed, it holds that
	\[
	\{ z \in (0, 1) \colon \q' (z) = 0 \} = \{ z \in (0, 1) \colon d_G (z) = 0 \} = (0, 1) \cap G, 
	\]
	which has strictly positive Lebesgue measure, as \(G\) has this property. 
	It is left to prove that \(\q\) is strictly increasing.
Let \(0 \leq x < y \leq 1\).
Recall that $G$ does not contain any open interval.
Therefore, the open set \(\{ z \in (x, y) \colon d_G (z) > 0 \} = (x, y) \cap G^c\) is nonempty and hence contains an open nonempty interval.
This yields that
\[
\q (y) - \q (x) = \int_x^y d_G (z) dz > 0
\]
and therefore, completes the proof.
\end{proof}

\begin{remark}
It is worth noting that there even exist $C^\infty$-functions with the properties from Lemma~\ref{lem:140624a1}.
To construct such a function, replace the function $d_G$ in the above proof by any $C^\infty$-function $f\colon[0,1]\to[0,1]$ such that $\{x\in[0,1]\colon f(x)=0\}=G$
(see, e.g., \cite[Exercise 3.18]{pugh}).
\end{remark}

Finally, let \(\mathsf{W}\) be a Brownian motion starting in a point \(x_0 \in (0, 1)\) that is absorbed\footnote{We could also take (instantaneous or slow) reflecting boundaries.}
in both \(0\) and \(1\), and let $\q$ be any function as described in Lemma~\ref{lem:140624a1}.
By virtue of \cite[Exercise~VII.3.18]{RY}, \(\Y \triangleq \q (\mathsf{W})\) is a general diffusion with scale function \(\s\triangleq \q^{-1}\). Furthermore, by the property of \(\q\) as described in Lemma~\ref{lem:140624a1}, \(\Y\) is a semimartingale
(see \cite[Section~5]{CinJPrSha}). Consequently, \(\Y\) is a general diffusion semimartingale that does not have the RP, as part~(b) of Theorem~\ref{theo: complete} is violated.
\end{example}

The remainder of this section is dedicated to a study of the extreme points of certain semimartingale problems that are connected to general diffusion semimartingales. We first recall the definition, where we adapt the concept from the monograph \cite{JS} by J.~Jacod and A.~N.~Shiryaev to our path-continuous setting.

\smallskip 
In view of Standing Assumption~\ref{SA: 1}, let \((B, C)\) be the semimartingale characteristics of~\(\X\). Of course, the semimartingale characteristics are not unique in the pathwise sense, so that we have to take a ``good version'' in the sense of \cite[Proposition~II.2.9]{JS}. 

\begin{definition}[Semimartingale problem]
	A probability measure \(\Q\) on \((\Omega, \cF, (\cF_t)_{t \geq 0})\) is said to be a solution to the {\em semimartingale problem} \((B, C, x_0)\), if \(\Q (\X_0 = x_0) = 1\) and \(\X\) is a continuous \(\Q\)-semimartingale with semimartingale characteristics \((B, C)\).
	The set of solutions is denoted by \(\mathscr{S} (B, C, x_0)\).
\end{definition} 

\begin{lemma}[\text{\cite[Corollary~III.2.8]{JS}}]
	The set \(\mathscr{S} (B, C, x_0)\) is convex. 
\end{lemma}

The candidate characteristics \((B, C)\) are taken such that \(\P_{x_0}\in \mathscr{S} (B, C, x_0)\). However, the measure \(\P_{x_0}\) is not necessarily the only solution to the semimartingale problem \((B, C, x_0)\), as the following example illustrates. 

\begin{example} \label{ex: SBM}
	Let \(\Y\) be the sticky Brownian motion
with a fixed stickiness parameter $\rho>0$
as described in Example~\ref{ex: sticky BM}.
	Evidently, \(\Y\) is a semimartingale and its characteristics \((B^\Y, C^\Y)\) are given by 
	\[
	B^\Y = 0, \qquad C^\Y = \int_0^\cdot \1_{\{\Y_s \not = 0\}} ds.
	\]
	Under the law \(\mathcal{L}aw\, (\Y)\) of the process \(\Y\), the coordinate process \(\X\) is a semimartingale and its characteristics \((B, C)\) are given by 
	\[
	B = 0, \qquad C = \int_0^\cdot \1_{\{\X_s \not = 0 \}} ds.
	\]
	By definition, \(\mathcal{L}aw\, (\Y) \in \mathscr{S} (B, C, x_0)\) but it is not the only solution to this semimartingale problem.
Indeed, also the Wiener measure with starting value \(x_0\)
and the laws of sticky Brownian motions starting in $x_0$ with all other values for the stickiness parameter (and convex combinations of the latter) are solutions.
These observations also show that \(| \mathscr{S} (B, C, x_0)| = \infty\).
\end{example}

Of course, the Wiener measure is the unique solution to the semimartingale problem associated to \(B_t = 0\) and \(C_t = t\). This illustrates that the particular choice of the semimartingale characteristics is crucial for the uniqueness question. Regardless of the version of the characteristics taken into consideration, general diffusion semimartingales
with (locally) absolutely continuous scale functions
turn out to be extreme points. Our remaining program is a discussion of this observation. 

\smallskip 
As explained in Section~III.4.d of the monograph~\cite{JS}, extremality is closely connected to the RP. The next theorem provides the complete picture for general diffusion semimartingales.

\begin{theorem} \label{theo: extremal}
	Each of the equivalent items \textup{(a)--(c)} from Theorem~\ref{theo: complete} is equivalent to the following:
	\begin{enumerate}
		\item[\textup{(d)}] \(\P_{x_0}\) is extremal in \(\mathscr{S} (B, C, x_0)\).
	\end{enumerate}
\end{theorem}

\begin{proof} 
	This follows from our main Theorem~\ref{theo: complete}, Blumenthal's zero-one law and \cite[Therorem~III.4.29]{JS}. 	
\end{proof}

\begin{remark}
The question whether strong Markov solutions to non-well-posed semimartingale problems are extreme points has already been considered by J.~Jacod and M.~Yor~\cite{JY_77}. In \cite[Proposition~4.4]{JY_77}, they proved that all strong Markov solutions associated to the ``extended generator''
\[
\mathcal{L} f (x) \triangleq \frac{1}{2} \frac{|x|^{2 \alpha}}{(1 + |x|^\alpha)^2}\, f'' (x), \quad f \in C^\infty_c (\bR; \bR), \ \ \alpha \in (0, \tfrac{1}{2}), 
\] 
are extreme points of the corresponding semimartingale problem, which is well-known to be non-well-posed. The strong Markov solutions to this problem can be classified according to their behavior in the origin.
In the Remark on p.~123 from \cite{JY_77} it is asked whether \cite[Proposition~4.4]{JY_77} holds for more general semimartingale problems, even multidimensional and with jumps.
A similar problem was considered by D.~W.~Stroock and M.~Yor~\cite{StrYor}.
Section~4.5 from \cite{JY_77} hints that the \emph{strong} Markov property must be crucial for that question, as it provides an example of a Markov,  but not strong Markov, process (with jumps) that is \emph{not} an extreme point of the corresponding semimartingale problem.
A related observation was made in \cite[Example 2.13]{StrYor}, which gives a non-extremal solution to a one-dimensional path-continuous semimartingale problem that is strongly Markov but only for the non-right-continuous canonical filtration. 
In the introduction to \cite{StrYor} and \cite[Remark 2.15]{StrYor} it has been emphasized that this example does not rule out that solutions with the strong Markov property for the right-continuous canonical filtration need to be extremal. 
Theorem~\ref{theo: extremal} answers this question in terms of Example~\ref{ex: no RP} that provides a strong Markov semimartingale for the right-continuous canonical filtration that is {\em not extremal}. 
\end{remark}

\section{Proof of Theorem~\ref{theo: complete}} \label{sec: pf}

This section is dedicated to the proof of Theorem~\ref{theo: complete}. We proceed in three steps. First, we show the equivalence of (b) and~(c), then the implication
(b)~\(\Longrightarrow\)~(a) and finally, (a)~\(\Longrightarrow\)~(b).

\newcommand*{\genRNderiv}{\frac{d\q\phantom{_L}}{d\mu_L}}

\subsection{Proof of (b) \(\Longleftrightarrow\) (c)}
Essentially, the equivalence follows from Theorem~C.7 in~\cite{criensurusov_22}.
For completeness, we provide some details.
Let $\genRNderiv$ be the
\emph{generalized Radon--Nikodym derivative}
(see Appendix~C in \cite{criensurusov_22})
of the measure $d\q$ on $(\s(J^\circ),\mathcal B(\s(J^\circ)))$ induced by the function $\q$
w.r.t. the Lebesgue measure $\mu_L$.
This is defined as the
$d\q,d\mu_L$-a.e. unique
Borel function $f\colon\s(J^\circ)\to[0,\infty]$ such that
$f=\frac{d\q}{d\gamma}/\frac{d\mu_L}{d\gamma}$
$d\q,d\mu_L$-a.e., where $d\gamma\triangleq d\q+d\mu_L$.
Alternatively, $\genRNderiv$ is called the
\emph{Lebesgue derivative} of $d\q$ w.r.t. $\mu_L$
in the seminal monograph
\cite{Shiryaev}
by A.~N.~Shiryaev
(see the object ``$\tilde z/z$'' in formula~(29) of \cite[Section~III.9]{Shiryaev}).
By \cite[Theorem C.7]{criensurusov_22},
$\genRNderiv=q'$ $d\q,d\mu_L$-a.e. on $\s(J^\circ)$
(notice that $q'$ exists everywhere on $\s(J^\circ)$ up to a countable set, while the latter does not matter both under $\mu_L$ and under $d\q$, as $\q$ is continuous).
This yields that
\begin{align*}
\text{(b)}
&\;\;\Longleftrightarrow\;\;
\genRNderiv>0\;\;\mu_L\text{-a.e.}
\\&\;\;\Longleftrightarrow\;\;
d\mu_L\ll d\q \text{ on } \s (J^\circ) \phantom \int
\\&\;\;\Longleftrightarrow\;\;
d\mu_L\circ\varphi^{-1}\ll d\q\circ\varphi^{-1} \text{ on } \varphi (\s (J^\circ))  \phantom \int
\end{align*}
for any continuous strictly increasing function
$\varphi\colon\s(J^\circ)\to\bR$. Notice that 
\[
d \mu_L \circ \q^{-1} = d \s, \qquad d \q \circ \q^{-1} = d \mu_L. 
\]
Hence, taking $\varphi=\q$ in the above equivalence, we conclude that
$$
\text{(b)}
\;\;\Longleftrightarrow\;\;
d\s\ll d\mu_L
\text{ on }J^\circ
\;\;\Longleftrightarrow\;\;
\text{(c)},
$$
which completes the proof.\qed

\smallskip
Lastly, we remark that an alternative
proof of the equivalence
(b)~\(\Longleftrightarrow\)~(c)
is due to \cite[Exercise 5.8.54]{bogachev} or \cite[Exercise 3.21]{leoni}.
To our taste, the above proof appears to be somewhat simpler (cf. the hint in \cite{bogachev}).

\subsection{Proof of (b) \(\Longrightarrow\) (a)}
We adapt the strategy from the proof of \cite[Theorem~III.4.29]{JS}. For contradiction, assume that the RP fails. By \cite[Corollary~III.4.27]{JS}, there exists a
bounded martingale \(\Z\) with \(0 \leq \Z \leq 2\), \(\Z_0 = 1\), \(\P_{x_0}(\Z_\infty = 1) < 1\) and \(\langle \Z, \X^c \rangle = 0\).  
We define a new probability measure by the formula 
\[
\Q (A) \triangleq \E^{\P_{x_0}} \big[ \Z_\infty \1_A \big], \quad A \in \cF.
\] 
In the following, we distinguish three cases:
\begin{enumerate}
	\item[(i)] \(J = (0, \infty)\); 
	\item[(ii)] \(J = [0, \infty)\) and the origin is an
(instantaneously or slowly)
reflecting boundary for \((x \mapsto \P_x)\); 
	\item[(iii)] \(J = [0, \infty)\) and the origin is an absorbing boundary for \((x \mapsto \P_x)\).
\end{enumerate}
All other cases can be treated by similar methods. 
	
\smallskip
{\em Case~(i):} 
As part~(b) from Theorem~\ref{theo: complete} holds, the occupation time formula (\cite[Theorem~29.5]{kallenberg}) yields that 
\begin{align}\label{eq: no measure on zero sets}
	\int_0^\cdot \1_{\{\q'_+ (\s (\X_s)) = 0\}} d \langle \s (\X), \s (\X)\rangle_t = \int \1_{\{ \q'_+ (x) = 0\}}L^x_\cdot (\s (\X)) dx = 0.
\end{align}
Since \(\X = \q (\s (\X))\), using the generalized It\^o formula (\cite[Theorem~29.5]{kallenberg}) and the facts that \(\q\) is the difference of two convex functions
and $\s(\X)$ is a local $\P_{x_0}$-martingale (\cite[Corollary~V.46.15]{RW2}),
we obtain that
\begin{align}\label{eq: continuous martingale part}
	d \X^c_t = \q'_+ (\s (\X_t)) d \s (\X_t).
\end{align}
From \eqref{eq: no measure on zero sets} and \eqref{eq: continuous martingale part}, it follows that 
\begin{equation} \label{eq: formula continuous local martingale part 1}
	\begin{split}
	d \s (\X_t) &= \1_{\{ \q'_+ (\s (\X_t)) \not = 0\}}d \s (\X_t) \phantom \int
	\\&= \frac{\1_{\{ \q'_+ (\s (\X_t)) \not = 0\}}}{\q'_+ (\s (\X_t))} \, \q'_+ (\s (\X_t)) d \s (\X_t)
	\\&= 
	\frac{\1_{\{ \q'_+ (\s (\X_t)) \not = 0\}}}{\q'_+ (\s (\X_t))} \, d \X^c_t.
\end{split}
\end{equation}
Next, we use a martingale problem argument.
Let \(f\in C_b ((\s (0), \s (\infty)); \bR)\) be the difference of two convex functions such that \(d f'_+ = 2 g  d \m \circ \s^{-1}\) on \((\s (0),\s(\infty))\) for some \(g \in C_b ((\s (0), \s (\infty)); \bR)\). 
Using the generalized It\^o formula (\cite[Theorem 29.5]{kallenberg}), the fact that the semimartingale local time of a local martingale is jointly continuous (\cite[Theorem~VI.1.7]{RY}) and~\eqref{eq: formula continuous local martingale part 1}, we obtain
\begin{align*}
d f (\s (\X_t))
&=
f'_+ (\s (\X_t)) d \s (\X_t) + d \int L^{y}_t (\s (\X)) g (y)\,\m\circ\s^{-1} (dy)
\\
&=
f'_+ (\s (\X_t)) \frac{\1_{\{ \q'_+ (\s (\X_t)) \ne 0\}}}{\q'_+ (\s (\X_t))} \, d \X^c_t + d \int L^{y}_t (\s (\X)) g (y)\,\m\circ\s^{-1} (dy).
\end{align*}
Recall
that \(\s (\X)\) is a general diffusion on natural scale with speed measure \(\m \circ \s^{-1}\), cf. \cite[Exercise VII.3.18]{RY}.
Thanks to the occupation time formula for general diffusions
on natural scale
(\cite[Theorem V.49.1]{RW2}), it holds that 
$$
\int L^{y}_\cdot (\s (\X)) g (y)\,\m\circ\s^{-1} (dy) = \int_0^\cdot g (\s (\X_s))\,ds.
$$
Putting these pieces together, we conclude that 
\begin{align*}
	d f (\s (\X_t)) - g (\s (\X_t)) dt = f'_+ (\s (\X_t)) \frac{\1_{\{ \q'_+ (\s (\X_t)) \not = 0\}}}{\q'_+ (\s (\X_t))} \, d \X^c_t.
\end{align*}
Because \(\langle \Z, \X^c\rangle = 0\), Girsanov's theorem (\cite[Theorem~III.3.11]{JS}) yields that the process \(\X^c\) is a local \(\Q\)-martingale. Consequently, also
\[
f (\s (\X)) - \int_0^\cdot g (\s (\X_s)) ds 
\]
is a local \(\Q\)-martingale. By the martingale problem for diffusions as given by \cite[Theorem~A.3, Lemma~A.4]{criensurusov_22}, we conclude that \(\Q = \P_{x_0}\). This, however, is a contradiction to \(\P_{x_0} (\Z_\infty = 1) < 1\). Summing up, we proved that the RP holds.\qed
\\

\smallskip 
{\em Case~(ii):}
The difference compared to Case~(i) is that the state space is not open, which requires more care to get a formula like~\eqref{eq: formula continuous local martingale part 1}.

Recall two facts: The space-transformed process \(\s (\X)\) is a semimartingale (as it is a time-change of Brownian motion, see \cite[Theorem~V.47.1]{RW2} and \cite[Corollary~10.12]{Jacod}), and \(\q \colon \s(J) = [\s(0), \s(\infty)) \to \bR\) is a continuous function such that \(\q'_+\) exists everywhere on \(\s(J)\) as a right-continuous function of (locally) finite variation. The latter fact follows from our assumptions that \(\X\) is a semimartingale and the origin is reflecting, cf. \cite[Section~5]{CinJPrSha}. By virtue of these facts, we can apply the generalized It\^o formula from \cite[Lemma~B.25]{criensurusov_22} and obtain that 
\begin{align*}
	d \X_t = d \q (\s (\X_t)) = \q'_+ (\s (\X_t)) d \s (\X_t) + \frac{1}{2}\, d \int_{(\s (0), \s (\infty))} L^{y-}_t (\s (\X)) \q'_+ (dy), 
\end{align*}
where \(\q'_+ (dx)\) is the second derivative measure associated to \(\q'_+\), i.e., 
\[
\q'_+ ( (a, b]) = \q'_+ (b) - \q'_+ (a), \quad a, b \in \s (J^\circ), \, a < b.
\] 
In view of this equation, we get that 
\begin{align} \label{eq: CLM part reflecting}
	d \X^c_t = \q'_+ (\s (\X_t)) d \s (\X)^c_t.
\end{align}
Due to \cite[V.47.23 (ii)]{RW2}, the process 
\begin{align*}
\s (\X) - \tfrac{1}{2}L^{\s (0)} (\s (\X))
\end{align*}
is a local martingale. Hence, 
\[
\s (\X)^c = \s (\X) - \s(x_0) - \tfrac{1}{2} L^{\s (0)} (\s (\X)).
\] 
As \eqref{eq: no measure on zero sets} also holds in the current setting, we obtain that 
\begin{equation}  \label{eq: continuous local martingale part 3}
	\begin{split}
	d \s (\X_t) &= d \s (\X)^c_t + \tfrac{1}{2} dL^{\s (0)}_t (\s (\X)) 
	\\&= \1_{\{\q'_+ (\s (\X_t)) \not = 0\}} d \s (\X)^c_t + \tfrac{1}{2} dL^{\s (0)}_t (\s (\X)) 
		\\&= \frac{\1_{\{\q'_+ (\s (\X_t)) \not = 0\}}}{\q'_+ (\s(\X_t))} \q'_+ (\s (\X_t)) d \s (\X)^c_t + \tfrac{1}{2} dL^{\s (0)}_t (\s (\X)) 
	\\&= \frac{\1_{\{\q'_+ (\s (\X_t)) \not = 0\}}}{\q'_+ (\s(\X_t))} d \X^c_t + \tfrac{1}{2} dL^{\s (0)}_t (\s (\X)).
\end{split}
\end{equation} 
Let \(f\in C_b ([\s (0), \s (\infty)); \bR)\) be such that the right-derivative \(f'_+\) exists as a right-continuous function \([\s (0), \s (\infty))\to \bR\) of locally finite variation, and \(d f'_+ = 2 g  d \m \circ \s^{-1}\) on \((\s (0), \s(\infty))\) and \(f'_+ (\s(0)) = 2 g (\s (0)) \m (\{0\})\) for some \(g \in C_b ([\s (0), \s (\infty)); \bR)\). 
Again using the generalized It\^o formula from \cite[Lemma B.25]{criensurusov_22}, the fact that the local time process \(L (\s (\X))\) is space-continuous on \((\s (0), \s (\infty))\) by \cite[Theorem V.49.1]{RW2} and formula~\eqref{eq: continuous local martingale part 3},
we obtain that 
\begin{equation*}
\begin{split}
d f (\s (\X_t)) &= f'_+ ( \s(\X_t)) d \s (\X_t) + d \int_{ (\s(0), \s(\infty))} L^{y-}_t (\s(\X)) g (y) \,\m\circ\s^{-1}(dy)
\\
&=
f'_+ (\s (\X_t)) \Big(\frac{\1_{\{\q'_+ (\s (\X_t)) \not = 0\}}}{\q'_+ (\s(\X_t))} d \X^c_t + \tfrac{1}{2} d L^{\s (0)}_t (\s(\X)) \Big)
\\
&\hspace{3.675cm}
+d \int_{ (\s(0), \s(\infty))} L^{y}_t (\s(\X)) g (y) \,\m\circ\s^{-1}(dy)
\\
&=
f'_+ (\s (\X_t))\frac{\1_{\{\q'_+ (\s (\X_t)) \not = 0\}}}{\q'_+ (\s(\X_t))} d \X^c_t + g (\s (0)) \m (\{0\}) d L^{\s (0)}_t (\s(\X))
\\
&\hspace{3.675cm}
+d \int_{ (\s(0), \s(\infty))} L^{y}_t (\s(\X)) g (y) \,\m\circ\s^{-1}(dy)
\\
&=
f'_+ (\s (\X_t))\frac{\1_{\{\q'_+ (\s (\X_t)) \not = 0\}}}{\q'_+ (\s(\X_t))} d \X^c_t 
+ d \int_{[\s(0), \s(\infty))} L^{y}_t (\s(\X)) g (y)\,\m\circ\s^{-1}(dy)
\\
&=
f'_+ (\s (\X_t))\frac{\1_{\{\q'_+ (\s (\X_t)) \not = 0\}}}{\q'_+ (\s(\X_t))} d \X^c_t + g (\s (\X_t))\,dt,
\end{split}
\end{equation*}
where in the last step we apply the occupation time formula for general diffusions on natural scale to $\s(\X)$ (see \cite[Theorem 136, p.~160]{freedman}\footnote{A delicate point:
When $0$ is reflecting, the generalization of the occupation time formula as in \cite[Theorem V.49.1]{RW2} holds also with a nonvanishing $g(\s(0))$.
This is discussed in \cite[Theorem 136, p.~160]{freedman}.}).
From this point on, we can conclude precisely as in Case~(i).\qed 
\\

\smallskip 
{\em Case~(iii):}
Let \(f\in C_b ([\s (0), \s (\infty)); \bR)\) be such that, restricted to \((\s (0), \s (\infty))\), \(f\) is the difference of two convex functions and \(d f'_+ = 2 g  d\m \circ \s^{-1}\) on \((\s (0), \s (\infty))\) for some \(g \in C_b ([ \s (0), \infty); \bR)\) with \(g (\s (0)) = 0\).
By similar arguments as in Case (i), we obtain, a.s. for all \(t < \inf \{t \geq 0 \colon \X_t \not \in J^\circ\}\), 
\[
 f (\s (\X_t)) - \int_0^t g (\s (\X_s)) ds = \int_0^t f'_+ (\s (\X_s)) \frac{\1_{\{ \q'_+ (\s (\X_s)) \not = 0\}}}{\q'_+ (\s (\X_s))} \, d \X^c_s.
\] 
By virtue of \cite[Exercise IV.1.48]{RY}, and using that \(\X\) is absorbed in its attainable boundary point and \(g (\s (0)) = 0\), it follows that 
\[
f (\s (\X)) - \int_0^\cdot g (\s (\X_s)) ds
\] 
is a (globally defined) local \(\Q\)-martingale. Now, we can conclude as in Case~(i).\qed

	\subsection{Proof of (a) \(\Longrightarrow\) (b)}
	To ease our notation, we assume that \(J\) is either \((0, \infty)\) or \([0, \infty)\). This allows us to work only with the right derivative \(\q'_+\).
	Suppose that the RP holds. 
	Recall again that the space-transformed process \(\s (\X)\) is a semimartingale (as it is a time-change of Brownian motion, see \cite[Theorem~V.47.1]{RW2} and \cite[Corollary~10.12]{Jacod}).
	Then, by the RP, there exists a process \(H \in L^2_\textup{loc} (\X^c)\) such that 
	\[
	\int_0^{\cdot} \1_{\{ \q'_+ (\s (\X_s)) = 0\}} \, d \s (\X)^c_s = \int_0^\cdot H_s \, d \X^c_s. 
	\] 
Using \eqref{eq: continuous martingale part} or \eqref{eq: CLM part reflecting}, we get that 
	\begin{align*}
		\int_0^{\cdot} \1_{\{ \q'_+ (\s (\X_s)) = 0\}} \, d \s (\X)^c_s = \int_0^\cdot H_s \q'_+ (\s (\X_s)) \, d \s (\X)^c_s. 
	\end{align*}
	Subtracting one side from the other and then computing the quadratic variation, it follows that a.s. \(d \langle \s (\X), \s (\X) \rangle\)-a.e.
	\[
	\1_{\{ \q'_+ (\s (\X)) = 0\}} = H \q'_+ (\s (\X)).
	\] 
	Evidently, this yields that a.s. \(d \langle \s (\X), \s (\X) \rangle\)-a.e.
	\[
	\q'_+ (\s (\X)) \not = 0.
	\] 
	Take \(p \in \s (J^\circ \setminus \{x_0\})\) and set \(T_p \equiv T_p (\s (\X)) \triangleq \inf \{t \geq 0 \colon \s (\X_t) = p \}\). Then, by the occupation time formula (\cite[Theorem~29.5]{kallenberg}), a.s. on \(\{ T_p < \infty\}\),
	\begin{align*}
		\int L_{T_p}^x (\s (\X)) \1_{\{ \q'_+ (x) = 0\}} dx = \int_0^{T_p} \1_{\{ \q'_+ (\s (\X_s)) = 0\}} d \langle \s (\X), \s (\X) \rangle_s = 0.
	\end{align*}
Thanks to \cite[Theorem~V.49.1]{RW2} and \cite[Corollary~29.18]{kallenberg},
a.s. on \(\{T_p < \infty\}\), \(L^x_{T_p} (\s (\X)) > 0\) for all \(x \in (\s(x_0) \wedge p, \s(x_0) \vee p)\).
By the regularity of the diffusion and because \(x_0\) is non-absorbing (cf. \cite[Section~2.2]{freedman}),
it holds that \(\P_{x_0} (T_p < \infty) > 0\).
Altogether, we conclude that 
	\[
	\mu_L ( \{ x \in (\s(x_0) \wedge p, \s(x_0) \vee p) \colon \q'_+ (x) = 0 \} ) = 0.
	\] 
	As \(p \in \s (J^\circ \setminus \{x_0\})\) was arbitrary, this implies that 
	\[
	\mu_L (\{ x \in \s (J^\circ) \colon \q'_+ (x) = 0\} ) = 0.
	\] 
	The proof of the implication (a) \(\Longrightarrow\) (b) is complete.\qed

\bibliographystyle{plain}

\end{document}